\documentclass[12pt,reqno]{amsart}

\usepackage{latexsym}
\usepackage{amssymb}
\usepackage{graphics}
\usepackage{graphicx}

\renewcommand{\epsilon}{\varepsilon}

\newcommand{\PP}{{\mathbb P}}

\newcommand{\G}{{\mathbb G}}

\newcommand{\C}{{\mathbb C}}
\newcommand{\Q}{{\mathbb Q}}
\newcommand{\Z}{{\mathbb Z}}

\newcommand{\CP}{\C\PP}

\newcommand{\rank}{{\operatorname{rank}}}
\newcommand{\fo}{{\operatorname{for}}}

\newcommand{\Pic}{{\operatorname{Pic}}}

\renewcommand{\phi}{\varphi}

\newcommand{\ocal}{\mathcal{O}}

\newtheorem{theo}{{Theorem}}[section]
\newtheorem{cor}[theo]{{Corollary}}
\newtheorem{lem}[theo]{{Lemma}}
\newtheorem{prop}[theo]{{Proposition}}

\newenvironment{rem}{\medskip\noindent{\it Remark:\/} }{\medskip}

\title[The Demailly-semple jet bundle]{On the Demailly-Semple jet bundles of hypersurfaces in $\CP^3$ }

\author{Jingzhou Sun }
\address{Department of Mathematics, Johns Hopkins University, Baltimore, MD
21218, USA} \email{jzsun@math.jhu.edu}

\thanks{Research  partially supported by NSF grant
 DMS-0901333.}

\date{\today}

\begin{document}

\begin{abstract}
Let $X$ be a smooth hypersurface of degree $d$ in $\CP^3$. By totally algebraic calculations, we prove that on the third Demailly-Semple jet bundle $X_3$ of $X$, the bundle $\ocal_{X_3}(1)$ is big for $d\geq 11$, and that on the fourth  Demailly-Semple jet bundle $X_4$ of $X$, the bundle $\ocal_{X_4}(1)$ is big for $d\geq 10$, improving a recent result of Diverio. 

\end{abstract}

\maketitle

 \tableofcontents
 \section{Introduction}
On the road to conquer the Kobayashi conjecture, more generally the Green-Griffiths conjecture, one fundamental idea is from Green-Griffiths's theorem(\cite{gg}). The idea is that the sections of jet differentials with values in a negative line bundle put restrains on entire curves $f : \C \rightarrow X$, therefore if we can produce enough such sections we will be able to prove the algebraic degeneracy of entire curves. Demaily(\cite{de}) generalized this idea to invariant jet differentials(theorem \ref{theo:ggde}). A very important advantage of this generalization is that the invariant jet differentials can be considered as direct images of line bundles on Demailly-Semple jet bundles, which is more algebraically computable.

Along this line of ideas, knowing some sections of invariant jet differentials, there are basically two directions. The first one is to algebraically analyze the base loci to show that the base loci are of small dimensions. In this direction, Demailly and Goul (\cite{dee}) showed that very generic hypersurfaces of degree $d\geq 21$ in $\CP^3$ is Kobayashi hyperbolic. Around the same time in \cite{mcm}, using different approach McQuillan showed that for $d\geq 36$ in $\CP^3$ generic hypersurfaces are Kobayashi hyperbolic as a corollary of his general theorem. The second direction is to use deformation methods(suggested by Siu\cite{siu}) to produce more sections and to show that base loci are of small dimensions. One explicit deformation method that Siu suggested is to use meromorphic vector fields to differentiate given sections. In this direction, Mihai P\v{a}un \cite{pa} showed that very generic hypersurfaces of degree $d\geq 18$ in $\CP^3$ is Kobayashi hyperbolic. In the same direction, S. Diverio, J. Merker, and E. Rousseau \cite{dmr} showed that in a generic hypersurface of degree $d\geq 2^{n^5}$ in $\CP^{n+1}$ every entire curve is algebraically degenerate.

Before one can choose from the two directions, the first key step is to get some sections of invariant jet differentials with values in a negative line bundle. Let $E_{k,m}T_X^*$ stand for the sheaf of invariant jet differentials
of order $k$ and total degree $m$ on a projective manifold $X$, and let $X_k$ denote the Demailly-Semple $k$-jet bundle of $X$, both of which will be defined in greater detail in section \ref{sec:basic}. Then $E_{2,m}T_X^*=(\pi_{k,0})_*\ocal_{X_k}(m)$, where $\pi_{k,o}:X_k\rightarrow X$ is the projection. When $X$ is of dimension 2 and $k=2$, $E_{2,m}T_X^*$ has a natural filtration
\begin{equation}\label{for:fil2}0\rightarrow S^mT_X^*\rightarrow E_{2,m}T_X^*\rightarrow E_{2,m-3}T_X^*\otimes K_X\rightarrow 0\end{equation}
so that one can calculate the Euler characteristic and further more show that $\ocal_{X_2}(1)$ is big when $X$ is a hypersurface of degree $d\geq 15$ in $\CP^3$ \cite{de}. Actually, for $k\geq 2$, $E_{k,m}T_X^*$ has a similar filtration up on $X_{k-2}$(proposition \ref{prop:fil}).

In \cite{di1},\cite{di2}and \cite{di3}, the holomorphic Morse inequalities were used to show that in $\CP^{n+1}$ hypersurfaces of certain degrees have sections of $k$-jet differentials with values in a negative curve. In particular, when $n=2$ it was showed in \cite{di3} that  $\ocal_{X_3}(1)$ and $\ocal_{X_4}(1)$ are big for $d\geq 12$

In this article, we will mainly prove two results in this direction of efforts. Both are  a little better than that in \cite{di3}.
\begin{theo}\label{theo:1}
Let $X$ be a hypersurface of degree $d\geq 11$ in $\CP^3$, then the line bundle $\ocal_{X_3}(1)$ on the Demailly-Semple 3-bundle $X_3$ of $X$ is big.
\end{theo}
 It was proved in \cite{me} that $\ocal_{X_4}(1)$ is big for $d\geq 9$ by considering the full algebra of Demailly invariants, but the method used there is not accessible to the author.
\begin{theo}\label{theo:2}
Let $X$ be a hypersurface of degree $d\geq 10$ in $\CP^3$, then the line bundle $\ocal_{X_4}(1)$ on the Demailly-Semple 4-bundle $X_4$ of $X$ is big.
\end{theo}
The main idea of the proofs of these two theorem is to apply the semistability of the cotangent bundle of $X$. Since in our estimations of dimensions of cohomology groups, we use inequalities from filtrations, which is somewhat coarse, we can not say that the two lower bounds are sharp. We hope that new techniques can be introduced to get better lower bounds.

This article is organized as follows. In section \ref{sec:basic}, we give definitions and results on jet differentials and Demailly-Semple jet bundles following those in \cite{de} and \cite{dee}. In section \ref{sec:semistable} we introduce the knowledge of semistable vector bundles we need. Then in section \ref{sec:3sem} and \ref{sec:4sem} we will prove theorem \ref{theo:1} and theorem \ref{theo:2}.

\textbf{Acknowledgement}:
The author would like to thank Professor Bernard Shiffman for introducing me this topic, for his continuous guidance and consistent patience in answering all my questions.  The author would also like to thank Professor Fedor Bogomolov and Yi Zhu  for helpful discussions. The author want also to thank Professor Simone Diverio and Professor Jo\"{e}l Merker for their suggestions on this article.

 \section{Jet differentials and Demailly-Semple jet bundle }\label{sec:basic}
In the terminology of
\cite{de}, a directed manifold is a pair $(X, V)$, where $X$ is a complex manifold
and $V\subset T_X$ a subbundle.
Let $(X, V)$ be a complex directed manifold, $
J_kV \rightarrow X$ is defined to be the bundle of $k$-jets of germs of curves $f: (\C,0)
\rightarrow X$ which are
tangent to V, i.e., such that $f'(t)\in  V_{f(t)}$ for all $t$ in a neighborhood of 0, together
with the projection map $f\rightarrow f(0)$ onto $X$. It is easy to check that $J_kV$ is actually a
subbundle of $J_kT_X$.  Let $\G_k$ be the group of germs of $k$-jet
biholomorphisms of $(\C, 0)$, that is, the group of germs of biholomorphic maps
$$t \rightarrow \phi(t)
= a_1t + a_2t^2 +\cdots +a_kt^k, \qquad a_1\in \C^*, a_j\in \C, \quad j > 2$$
in which the composition law is taken modulo terms $t^j$ of degree $j > k$. The
group $\G_k$ acts on the left on $J_kV$ by reparametrization, $(\phi,f)\rightarrow f\circ\phi$.

 Given a directed manifold $(X,V)$ with $\rank V=r$, let $\tilde{X}=\PP(V)$. The subbundle $\tilde{V}\subset T_{\tilde{X}}$ is defined by $$\tilde{V}_{x,[v]}=\{\xi\in T_{\tilde{X},(x,[v])}|\pi_*\xi\in \C\cdot v\}$$ for any $x\in X$ and any $v\in T_{X,x}\backslash \{0\}$. Let $T_{\tilde{X}|X}$ denote the relative tangent bundle with respect to the projection $\pi:\tilde{X}\rightarrow X$, we will be making use of the following exact sequences
\begin{equation}0\rightarrow T_{\tilde{X}|X}\rightarrow \tilde{V} \xrightarrow{\pi_*} \ocal_{\tilde{X}}(-1)\rightarrow 0\label{for:exact v}
\end{equation}\begin{equation}0\rightarrow \ocal_{\tilde{X}}\rightarrow \pi^*V\otimes \ocal_{\tilde{X}}(1) \rightarrow T_{\tilde{X}|X}\rightarrow 0
\end{equation}
From the above exact sequences we get
\begin{eqnarray}
c_1(\tilde{V})=c_1(T_{\tilde{X}|X})+c_1( \ocal_{\tilde{X}}(-1))=\pi^*c_1(V)+(r-1)c_1( \ocal_{\tilde{X}}(1))\label{for:det}
\end{eqnarray}
and when $\rank V=2$, we have $$T_{\tilde{X}|X}=\pi^* \det V\otimes \ocal_{\tilde{X}}(2)\\$$
Since each fiber is isomorphic to $\CP^{r-1}$, which we denote by $F_x$ for $x\in X$, and the restriction of $\ocal_{\tilde{X}}(m)$ to each fiber is isomorphic to $\CP^{r-1}(m)$, the function $h^i_m(x)=h^i(F_x, \CP^{(r-1)}_x(m))$ is constant on $X$. By Grauert's theorem, the higher direct images $R^i\pi_*\ocal_{\tilde{X}}(m)$ of $\ocal_{\tilde{X}}(m)$ under the projection $\pi:\tilde{X}\rightarrow X$ are locally free on $X$ for $i\geq 0$.

 In particular, when $m\geq 0$,
\begin{eqnarray}\pi_*\ocal_{\tilde{X}}(m)=S^mV^* \\ R^i\pi_*\ocal_{\tilde{X}}(m)=0 \quad\fo \quad i\geq 1\end{eqnarray}

Starting with a directed manifold $(X,V)=(X_0,V_0)$, we get a tower of directed manifolds $(X_k,V_k)$, called Demailly-Semple $k$-jet bundle of $X$, defined by $X_k=\tilde{X}_{k-1},V_k=\tilde{V}_{k-1}$. In particular, when $X$ is a hypersurface in $\CP^3$, we start with $(X, T_X)$

From now on, we will use the following notations $$\pi_k:X_k\rightarrow X_{k-1},\qquad T_{k,k-1}=T_{X_k|X_{k-1}}$$  $$\ocal_k(1)=\ocal_{X_k}(1), \qquad u_k=c_1(\ocal_k(1))$$ $$\pi_{i,j}:X_i\rightarrow X_j,\quad i>j$$
Note that the Picard group of $X_k$ is given by $$\Pic (X_k)=\Pic (X_{k-1})\oplus \Z[\ocal_k(1)]$$
and the cohomology ring $H^{\bullet}(X_k)$ is given by
\begin{equation}
H^{\bullet}(X_k)=H^{\bullet}(X_{k-1})[\ocal_k(1)]/(\ocal_k(1)^2+c_1(V_{k-1})\ocal_k(1)+c_2(V_{k-1}))
\end{equation}
\begin{theo}\cite{de}\label{theo:de 2,1}
The direct image sheaf $(\pi_{k,0})_*\ocal_{X_k}(m)$ on $X$ coincides
with the (locally free) sheaf $E_{k,m}V^*$ of $k$-jet differentials of weighted degree $m$, that
is, by definition, the set of germs of polynomial differential operators
\begin{equation} Q(f)=\sum_{\alpha_1\cdots\alpha_k}a_{\alpha_1\cdots\alpha_k}(f)(f')^{\alpha_1}(f'')^{\alpha_2}\cdots(f^{(k)})^{\alpha_k}
\end{equation}
on $J_kV$ (in multi-index notation, ($f')^{\alpha_1}=
((f_1')^{\alpha_{1,1}}(f_2')^{\alpha_{1,2}}\cdots(f_r')^{\alpha_{1,r}}) $, which are
moreover invariant under arbitrary changes of parametrization: a germ of operator
$Q\in  E_{k,m}V^*$ is characterized by the condition that, for every germ $f\in J_kV$ and
every germ $\phi \in \G_k$ ,
$$Q(f\circ\phi)
=
\phi'^mQ(f)\circ\phi$$
\end{theo}
On $X_2$, define the weighted line bundle $$\ocal_{X_2}(a_1, a_2)=\phi_{2*}(\ocal_{X_1}(a_1))\otimes \ocal_{X_2}( a_2)$$
The following lemma is part of lemma 3.3 in \cite{dee}
\begin{lem}\cite{dee}\label{lem:equal}Let $\rank V=2$,
for $m = a_1 + a_2 > 0$, there is an injection
$$(\pi_{2,0})_*(\ocal_{X_2}(a_1,a_2)) \rightarrow E_{2,m}V^* $$
and the injection is an isomorphism if $a_1-2a_2 < 0$.
\end{lem}

\begin{prop}\label{prop:fil}Similar to the filtration of $E_{2,m}T_X^*$(formula \ref{for:fil2}) for $\dim X=2$, the relative case $E_{2,m}V^*$ when $\rank V=2$ also has a filtration
\begin{equation}
\label{for:fil2j}0\rightarrow S^mV^*\rightarrow E_{2,m}V^*\rightarrow E_{2,m-3}V^*\otimes \det V^*\rightarrow 0
\end{equation}
\end{prop}
\begin{proof}
Write $m=3p+q$ for $0\leq q\leq 2$, then by lemma \ref{lem:equal}, $$E_{2,m}V^*=(\pi_{2,0})_*(\ocal_{X_2}(2p,p+q))$$
On the other hand, we have $(\pi_{2,1})_*(\ocal_{X_2}(2p,p+q))=S^{p+q}V_1^*\otimes \ocal_{X_1}(2p)$. By the exact sequence $$0\rightarrow \ocal_{X_1}(1)\rightarrow V_1^*\rightarrow T_{X_1|X}^*\rightarrow 0$$, we get exact sequence $$0\rightarrow \ocal_{X_1}(p+q)\rightarrow S^{p+q}V_1^*\rightarrow S^{p-1+q}V_1^*\otimes T_{X_1|X}^*\rightarrow 0$$, since $T_{X_1|X}^*=\pi_{1}^*\det V^*\otimes \ocal_{X_1}(-2)$, we get exact sequence
$$0\rightarrow \ocal_{X_1}(3p+q)\rightarrow S^{p+q}V_1^*\otimes \ocal_{X_1}(2p)\rightarrow S^{p-1+q}V_1^*\otimes \ocal_{X_1}(2p-2)\rightarrow 0$$
Since $R^1\pi_{1*}( \ocal_{X_1}(3p+q))=0$, pushing forward the exact sequence above, we get the claimed filtration.
\end{proof}
The following theorem as introduced in the introduction forms the foundation of our efforts.
\begin{theo}[\cite{de}] \label{theo:ggde}
Assume that there exist integers $k,m > 0$ and an ample line
bundle $L$ on $X$ such that $H^0(P^kV,\ocal_{P^kV} (m)\otimes \pi_{k,0}^*L^-1)\simeq H^0(X,E_{k,m}(V^*)\otimes L^{-1})$
has non zero sections $\sigma_1,\cdots, \sigma_N$. Let $Z \subset P_kV$ be the base locus of these sections.
Then every entire curve $f : \C \rightarrow X$ tangent to $V$ is such that $f^{[k]}(C)\subset Z$. In
other words, for every global $\G_k$-invariant polynomial differential operator $P$ with
values in $L^{-1}$, every entire curve $f$ must satisfy the algebraic differential equation
$P(f) = 0$.
\end{theo}
\section{semistablity and restriction theorem}\label{sec:semistable}
Let $\mathfrak{F}$ be a torsion-free coherent sheaf over a compact K\"{a}hler manifold $(M,\omega)$, let $c_1(\mathfrak{F})$ be the first Chern class of $\mathfrak{F}$. The $\omega-$degree of $\mathfrak{F}$ is defined to be $$\deg(\mathfrak{F})=\int_Mc_1(\mathfrak{F})\wedge \omega^{n-1}$$
The degree/rank ratio $\mu(\mathfrak{F})$ is defined to be$$\mu(\mathfrak{F})=\frac{\deg(\mathfrak{F})}{\rank (\mathfrak{F})}.$$
Recall that $\mathfrak{F}$ is $\omega-$semistableif for every coherent subsheaf $\mathfrak{F}'$, $0<\rank  \mathfrak{F}'$, we have $$\mu(\mathfrak{F})\leq \mu(\mathfrak{F}')$$
Two basic theorems of semistable vector bundles are as follows:
\begin{prop}[\cite{ko}]\label{theo:ko}
If $\mathfrak{F}$ is a $\omega$-semistable sheaf over a compact K\"{a}hler manifold $M$ such that $\deg(\mathfrak{F})<0$, then $\mathfrak{F}$ admits no nonzero holomorphic section.

\end{prop}

\begin{prop}[\cite{ko}]\label{theo:ko2}
Let $\mathfrak{F}$ be a torsion free coherent over a compact K\"{a}hler manifold $(M, g)$.Then

 (a)  Let $\mathfrak{L}$ be a line bundle over $M$. Then $\mathfrak{F}\otimes \mathfrak{L}$ is $\omega$-semistable if and only if $\mathfrak{F}$ is $\omega$-semistable.

 (b)$\mathfrak{F}$ is $\omega$-semistable if and only if its dual $\mathfrak{F}^*$ is $\omega$-semistable
\end{prop}

If we have an ample line bundle $H$, we can define $H-$semistability just to be $\omega_H-$semistability, where $\omega_H$ is a positive $(1,1)-$form in the first Chern class of $H$. And this definition can be generalized to a big and nef line bundle. The following theorem is due to Tsuji
\begin{theo}[\cite{ts}]\label{theo:ts}
Let $X$ be a smooth minimal algebraic variety over $\C$. Then the tangent bundle $T_X$ is $K_X$-semistable.
\end{theo}
\begin{rem}In particular, when $X$ is a hypersurface of degree $d\geq 5$ in $\CP^3$, $T_X$ is $K_X$-semistable. By proposition \ref{theo:ko2}, this also implies that $T_X^*=\Omega_X^1$ is $K_X$-semistable.
\end{rem}
When $M$ is a curve, the K\"{a}hler form $\omega$ does not appear in the definition of semistability, therefore we have absolute semistability.

The following restriction theorem is crucial in our estimations.
\begin{theo}[\cite{f}]\label{theo:f}
Let $X$ be a $n$-dimensional normal projective subvariety in $\PP^n$over the
algebraically closed field $k$ of characteristic 0. Let $\xi$ be a semistable torsion free $\ocal _X$-module of rank $r$ and $e$,$c$ integers, $1\leq c\leq n-1$, such that
$$\frac{\left(\begin{array}{ccc}n+e\\e\end{array}\right)-ce-1}{e}>\deg(X)\max (\frac{r^2-1}{4},1).$$
Then for a general complete intersection $T=H_1\cap\cdots\cap H_c,\quad H_i\in |\ocal_X(e)|$, the restriction $\xi|_Y:=\xi\otimes \ocal_Y$ is semistable on $Y$.
\end{theo}
\begin{rem}
When $X$ is a hypersurface of degree $d\geq 5$ in $\CP^3$, this theorem implies that for $e\geq 2d$, general curve $Y\in |\ocal_X(e)|$, the restriction $\Omega_X^1|_Y$ is semistable
\end{rem}
\section{The Third Semple Jet Bundle}\label{sec:3sem}
\subsection{Euler Characteristic}
Since we want to study $\ocal_3(1)$ on $X_3$, we denote $$\ocal_3(a_1,a_2,a_3)=\pi^*_{3,1}\ocal_1(a_1)\otimes \pi^*_3\ocal_2(a_2)\otimes \ocal_3(a_3),\qquad \ocal_3(a_2, a_3)=\pi^*_3\ocal_2(a_2)\otimes \ocal_3(a_3))$$

Since we have injection $\ocal_3(b,1)\rightarrow \ocal_3(1+b)$, where $b>0$, to show that $\ocal_3(1)$ is big, it suffices to show that $\ocal_3(b,1)$ is big.

In the following, we will first show that $\ocal_3(2,1)$ is big for $d\geq 12$. For $d=11$, we will modify the method for $\ocal_3(2,1)$ by allowing $b$ varying from in the interval $[2,\infty)$. One can always consider $b$ as a rational number so that $\ocal_3(b,1)$ make sense as $\Q-$line bundle.

\smallskip

First we calculate the Euler characteristic of $\ocal_3(2n,n)$. Since the fiber of the projection $\pi_3$ is $\CP^1$, the restriction of $\ocal_3(2n,n)$ is $\ocal_{\CP^1}(n) $, and for $n>0$, $H^1(\CP^1,O(n))=0$, so we have $$R^i\pi_{3*}\ocal_3(2n,n)=0, \qquad i\geq 1$$therefore $$H^i(X_3,\ocal_3(2n,n))=H^i(X_2,\pi_{3*}\ocal_3(2n,n)), i\geq 0$$ in particular $$\chi(\ocal_3(2n,n))=\chi(\pi_{3*}\ocal_3(2n,n))$$

It is clear that $\pi_{3*}\ocal_3(2n,n)=S^nV_2^*\otimes \ocal_2(2n)$. From the exact sequence $$0\rightarrow \ocal_2(1)\rightarrow V_2^* \rightarrow T_{2,1}^* \rightarrow 0$$ we see that $S^nV_2^*$ has a filtration with graded bundle $$Gr^{\bullet}(S^nV_2^*)=\bigoplus_{0\leq k \leq n}\ocal_2(n-k)\otimes (T_{2,1}^*)^{\otimes k}$$
Therefore, by substituting $T_{2,1}^*=\pi_2^*\det V_1^*\otimes \ocal_2(-2)$, $S^nV_2^*\otimes \ocal_2(2n)$ has a filtration with graded bundle $$Gr^{\bullet}(S^nV_2^*\otimes \ocal_2(2n))=\bigoplus_{0\leq k \leq n}\ocal_2(3n-3k)\otimes \pi_2^*(\det V_1^*)^{\otimes k}$$

We thus have the following formula for the Euler characteristic of $\ocal_3(2n,n)$:
\begin{equation}
\chi(\ocal_3(2n,n))=\sum_{k=0}^n\chi(\ocal_2(3n-3k)\otimes \pi_2^*(\det V_1^*)^{\otimes k})
\end{equation}
To calculate $\chi(\ocal_2(3n-3k)\otimes \pi_2^*(\det V_1^*)^{\otimes k})$, we push it to $X_1$. Since $\ocal_2(3n-3k)\otimes \pi_2^*(\det V_1^*)^{\otimes k}$ restricted to the fiber of $\pi_2:X_2\rightarrow X_1$ is $\ocal_{\CP^1}(3n-3k)$ and $3n-3k\geq 0$, we have $$R^i\pi_{2*}(\ocal_2(3n-3k)\otimes \pi_2^*(\det V_1^*)^{\otimes k})=0 , \quad i\geq 1$$ therefore $$H^i(X_2,\ocal_2(3n-3k)\otimes \pi_2^*(\det V_1^*)^{\otimes k})=H^i(X_1,\pi_{2*}(\ocal_2(3n-3k)\otimes \pi_2^*(\det V_1^*)^{\otimes k})),\quad i\geq 0$$
in particular $$\chi(\ocal_2(3n-3k)\otimes \pi_2^*(\det V_1^*)^{\otimes k})=\chi(\pi_{2*}(\ocal_2(3n-3k)\otimes \pi_2^*(\det V_1^*)^{\otimes k}))$$
Again, $\pi_{2*}(\ocal_2(3n-3k)\otimes \pi_2^*(\det V_1^*)^{\otimes k})=S^{3n-3k}V_1^*\otimes (\det V_1^*)^{\otimes k}$
and $S^{3n-3k}V_1^*$ has a filtration with graded bundle $$Gr^{\bullet}(S^{3n-3k}V_1^*)=\bigoplus_{0\leq l \leq 3n-3k}\ocal_1(3n-3k-l)\otimes (T_{1,0}^*)^{\otimes l}$$
Plugging in the equations $$T_{1,0}^*=\pi_1^*K_X\otimes \ocal_1(-2), \qquad \det V_1^*=\pi_1^*K_X\otimes \ocal_1(-1)$$
we get \begin{equation}
Gr^{\bullet}(S^{3n-3k}V_1^*\otimes (\det V_1^*)^{\otimes k})=\bigoplus_{0\leq l \leq 3n-3k}\ocal_1(3n-4k-3l)\otimes \pi_1^*(K_X)^{\otimes (k+l)}
\end{equation}
So we have the following equation $$\chi(\pi_{2*}(\ocal_2(3n-3k)\otimes \pi_2^*(\det V_1^*)^{\otimes k}))=\sum_{l=0}^{3n-3k}\chi(\ocal_1(3n-4k-3l)\otimes \pi_1^*(K_X)^{\otimes (k+l)})$$
In conclusion of the above analysis, we have
\begin{prop}
\begin{equation}\label{for:chi3}
\chi(\ocal_3(2n,n))=\sum_{k=0}^n\sum_{l=0}^{3n-3k}\chi(\ocal_1(3n-4k-3l)\otimes \pi_1^*(K_X)^{\otimes (k+l)})
\end{equation}

\end{prop}
From now on, by abusing of notations, we identify a cohomology class in $H^{\bullet}(X_j)$ with its image under $\pi_{i,j}^*$ in $H^{\bullet}(X_i)$ for $i>j$.

To simplify notation we write $L_{n,k,l}=\ocal_1(3n-4k-3l)\otimes \pi_1^*(K_X)^{\otimes (k+l)}$.  We also write $c_1=c_1(X)$ and $c_2=c_2(X)$, so

$$c_1(L_{n,k,l})=(3n-4k-3l)u_1-(k+l)c_1$$

Now we use Hirzebruch-Riemann-Roch formula to calculate $\chi(L_{n,k,l})$.

For any line bundle $L$ on $X_1$, the Hirzebruch-Riemann-Roch formula is
$$\chi(L)=\frac{1}{6}c_1^3(L)+\frac{1}{4}c_1^2(L)c_1(X_1)+\frac{1}{12}c_1(L)(c_1^2(X_1)+c_2(X_1))+\frac{1}{24}c_1(X_1)c_2(X_1)$$ Since to show that $\ocal_3(2n,n)$ is big, we need to show that the coeffecient of $n^5$ in $H^0(X_3,\ocal_3(2n,n))$ is positive, and observe that the only contribution to the coefficient of $n^5$ of $\chi(\ocal_3(2n,n))$ after summing up is from the first term in the preceding formula, we only need to calculate the first term
\begin{eqnarray}
c_1^3(L_{n,k,l})&=&((3n-4k-3l)u_1-(k+l)c_1)^3\\&=&(3n-4k-3l)^3u_1^3-3(3n-4k-3l)^2(k+l)u_1^2c_1\\&+&3(3n-4k-3l)(k+l)^2u_1c_1^2-(k+l)^3c_1^3
\end{eqnarray}
Plugging the following equations
\begin{eqnarray}
u_1^3=c_1^2-c_2&,&\qquad u_1^2c_1=-c_1^2\\
u_1c_1^2=c_1^2&,&\qquad c_1^3=0
\end{eqnarray}

We get
\begin{eqnarray}
c_1^3(L_{n,k,l})&=&(3n-4k-3l)^3(c_1^2-c_2)+3(3n-4k-3l)^2(k+l)c_1^2\\&+&3(3n-4k-3l)(k+l)^2c_1^2
\end{eqnarray}
Therefore
\begin{eqnarray}
\chi(\ocal_3(2n,n))&=&\sum_{k=0}^n\sum_{l=0}^{3n-3k}\frac{1}{6}[(3n-4k-3l)^3(c_1^2-c_2)+3(3n-4k-3l)^2(k+l)c_1^2\\&+&3(3n-4k-3l)(k+l)^2c_1^2]+O(n^4)\\&=&n^5(\frac{249}{60}c_2-c_1^2)+O(n^4)
\end{eqnarray}
\begin{theo}\label{theo:chi}
\begin{equation}
\chi(\ocal_3(2n,n))=n^5(\frac{249}{60}c_2-c_1^2)+O(n^4)
\end{equation}
\end{theo}

We now introduce semistablity of vector bundles, which will be crucial to estimate

$H^0(X_3, \ocal_3(2n,n))$.

\subsection{Estimation of $H^2$}
Now that we know $\chi(\ocal_3(2n,n))$, since $\chi=h^0-h^1+h^2-h^3+h^4-h^5$, to show that $h^0(X_3, \ocal_3(2n,n))$ has positive coefficient in $n^5$, we need to calculate $h^2(X_3, \ocal_3(2n,n))$ and $h^4(X_3, \ocal_3(2n,n))$.

Using again the filtration of $\pi_{3*}\ocal_3(2n,n)$, we see that
\begin{eqnarray}
h^2(X_3,\ocal_3(2n,n))=h^2(X_2,\pi_{3*}\ocal_3(2n,n))\leq \sum_{k=0}^nh^2(X_2, \ocal_2(3n-3k)\otimes \pi_2^*(\det V_1^*)^{\otimes k})\\
h^4(X_3,\ocal_3(2n,n))=h^4(X_2,\pi_{3*}\ocal_3(2n,n))\leq \sum_{k=0}^nh^4(X_2, \ocal_2(3n-3k)\otimes \pi_2^*(\det V_1^*)^{\otimes k})
\end{eqnarray}
Since $\dim X_1=3$, we have $$h^4(X_2,\ocal_2(3n-3k)\otimes \pi_2^*(\det V_1^*)^{\otimes k})=h^4(X_1,\pi_{2*}(\ocal_2(3n-3k)\otimes \pi_2^*(\det V_1^*)^{\otimes k}))=0$$
\smallskip
Now we need to calculate $h^2(X_2, \ocal_2(3n-3k)\otimes \pi_2^*(\det V_1^*)^{\otimes k})$

Pushing forward onto $X_1$, we have $$h^2(X_2, \ocal_2(3n-3k)\otimes \pi_2^*(\det V_1^*)^{\otimes k})=h^2(X_1, S^{3n-3k}V_1^*\otimes (\det V_1^*)^{\otimes k})$$
From the filtration of $S^{3n-3k}V_1^*\otimes (\det V_1^*)^{\otimes k}$ we showed above, we see that
\begin{equation}
h^2(X_1, S^{3n-3k}V_1^*\otimes (\det V_1^*)^{\otimes k})\leq \sum_{l=0}^{3n-3k}h^2(X_1, \ocal_1(3n-4k-3l)\otimes \pi_1^*(K_X)^{\otimes (k+l)})
\end{equation}
To calculate $h^2(X_1, \ocal_1(3n-4k-3l)\otimes \pi_1^*(K_X)^{\otimes (k+l)})$, we push them further onto $X$, then we have two situations
$$\pi_{1*}\ocal_1(3n-4k-3l)\otimes \pi_1^*(K_X)^{\otimes (k+l)}=\{\begin{array}{clcr}0& 3n-4k-3l< 0\\S^{3n-4k-3l}T_X^*\otimes (k+l)K_X& 3n-4k-3l\geq 0 \end{array} $$
\textbf{\textit{Case 1.}} when $3n-4k-3l\geq -1$, $R^i\pi_{1*}\ocal_1(3n-4k-3l)\otimes \pi_1^*(K_X)^{\otimes (k+l)}=0$ for $i\geq 1$, we have $$
h^2(X_1, \ocal_1(3n-4k-3l)\otimes \pi_1^*(K_X)^{\otimes (k+l)})=h^2(X, S^{3n-4k-3l}T_X^*\otimes (k+l)K_X), \quad 3n-4k-3l\geq 0
$$
and $$h^2(X_1, \ocal_1(3n-4k-3l)\otimes \pi_1^*(K_X)^{\otimes (k+l)})=0,\quad 3n-4k-3l=-1$$
First we show the following theorem
\begin{theo}
When $3n-4k-3l\geq 0$ and $k+l\geq 2$, or $3n-4k-3l\geq 1$ and $k+l\geq 1$, we have
\begin{equation}
h^2(X, S^{3n-4k-3l}T_X^*\otimes (k+l)K_X)=0
\end{equation}
\end{theo}
\begin{proof}
By Serre duality theorem, we have $$h^2(X, S^{3n-4k-3l}T_X^*\otimes (k+l)K_X)=h^0(X, (S^{3n-4k-3l}T_X^*\otimes (k+l-1)K_X)^*)$$
Since $K_X$ is ample, by theorem \ref{theo:ts} and proposition \ref{theo:ko2}, $S^{3n-4k-3l}T_X^*\otimes (k+l-1)K_X$ is $K_X$-semistable. By assumption $\deg (S^{3n-4k-3l}T_X^*\otimes (k+l-1)K_X)^*<0$, therefore by proposition \ref{theo:ko} and \ref{theo:ko2} $$h^0(X, (S^{3n-4k-3l}T_X^*\otimes (k+l-1)K_X)^*)=0$$
\end{proof}

\textbf{\textit{Case 2.}} When $3n-4k-3l\leq -2$, by the Leray spectral sequence of the projection $\pi_1:X_1\rightarrow X$,
\begin{equation}
h^2(X_1, \ocal_1(3n-4k-3l)\otimes \pi_1^*(K_X)^{\otimes (k+l)})=h^1(X, R^1\pi_{1*}(\ocal_1(3n-4k-3l)\otimes \pi_1^*(K_X)^{\otimes (k+l)}))
\end{equation}
\begin{prop}\label{for:R1} For $m\geq 2$,
$$R^1\pi_*(\ocal_1(-m))=(-K_X)\otimes S^{m-2}T_X$$
\end{prop}
\begin{proof}
First $h^1(\CP^1_x, \ocal_1(-m)|_{\CP^1_x})$ considered as a function on $X$ is constant, where $\CP^1_x$ is the fiber of $\pi_1$ at $x\in X$, since $\ocal_1(-m)|_{\CP^1_x})=\ocal(-m)$. So by Grauert's theorem, $R^1\pi_*(\ocal_1(-m))$ is locally free on $X$.

Next we claim that $R^1\pi_{1*}T_{X_1|X}^*=\ocal_X$

To see this, consider the exact sequence$$0\rightarrow T_{X_1|X}^*\rightarrow \pi_1^*T_X^*\otimes \ocal_1(-1)\rightarrow \ocal_{X_1}\rightarrow 0$$
Pushing forward to $X$ and noticing that $\pi_{1*}(\pi_1^*T_X^*\otimes \ocal_1(-1))=0$ and $R^1\pi_{1*}T_{X_1|X}^*(\pi_1^*T_X^*\otimes \ocal_1(-1))=0$, and since $\pi_{1*}\ocal_{X_1}=\ocal_X$, we have exact sequence $$0\rightarrow \ocal_X\rightarrow R^1\pi_{1*}T_{X_1|X}^*\rightarrow 0$$
Thus the claim is proved.

By applying the Serre duality theorem on the fibers of the projection $\pi_1:X_1\rightarrow X$, we see that the natural pairing $$\pi_*(T_{X_1|X}^*-\ocal_1(-m))\times R^1\pi_*(\ocal_1(-m))\rightarrow R^1\pi_{1*}T_{X_1|X}^*=\ocal_X$$ is a perfect pairing, therefore we have
\begin{eqnarray}
R^1\pi_*(\ocal_1(-m))&=&(\pi_*(T_{X_1|X}^*-\ocal_1(-m)))^*\\&=&(\det T_X^*\otimes S^{m-2}T_X^*)^*\\&=&(-K_X)\otimes S^{m-2}T_X
\end{eqnarray}
\end{proof}

We thus get that when $3n-4k-3l\leq -2$,
\begin{equation}
R^1\pi_{1*}(\ocal_1(3n-4k-3l)\otimes \pi_1^*(K_X)^{\otimes (k+l)})=S^{4k+3l-3n-2}T_X\otimes ((k+l-1)K_X)
\end{equation}

To calculate $h^1(X, S^{4k+3l-3n-2}T_X\otimes ((k+l-1)K_X))$, our strategy is: first we calculate the Euler characteristic, then estimate $h^0$ and $h^2$, after that we will be able to get a good estimation of $h^1$.

Now we use the Hirzebruch-Riemann-Roch formula to calculate the Euler characteristic of $S^{4k+3l-3n-2}T_X\otimes ((k+l-1)K_X)$

By the identity $T_X^*=T_X\otimes K_X$, we have $$S^{4k+3l-3n-2}T_X\otimes ((k+l-1)K_X)=S^{4k+3l-3n-2}T_X^*\otimes ((3n-3k-2l+1)K_X)$$

Observe that $S^{4k+3l-3n-2}T_X^*\otimes ((3n-3k-2l+1)K_X)$ can be considered as direct image of $\ocal_1(4k+3l-3n-2)\otimes \pi_1^*((3n-3k-2l+1)K_X)$ under the projection $\pi_1: X_1\rightarrow X$, therefore by the Hirzebruch-Riemann-Roch formula on $X_1$, we get
\begin{eqnarray}
\chi(S^{4k+3l-3n-2}T_X^*\otimes ((3n-3k-2l+1)K_X))=
\\\frac{1}{6}[(4k+3l-3n-2)^3(c_1^2-c_2)
+3(4k+3l-3n-2)^2(3n-3k-2l+1)c_1^2\\+3(4k+3l-3n-2)(3n-3k-2l+1)^2c_1^2]
+O(n^2)
\end{eqnarray}
summing up over all suitable $k$ and $l$, we get
\begin{prop}\label{prop:h2}
\begin{eqnarray}
\sum_{\begin{array}{clcr}0\leq k\leq n,\quad 0\leq l\leq 3n-3k\\3n-4k-3l\leq -2
\end{array}}\chi(S^{4k+3l-3n-2}T_X^*\otimes ((3n-3k-2l+1)K_X))\\
=(\frac{2013}{1536}c_1^2-\frac{2073}{480}c_2)n^5+O(n^4) \qquad \qquad \qquad
\end{eqnarray}
\end{prop}

To calculate $h^0(X, S^{4k+3l-3n-2}T_X\otimes ((k+l-1)K_X))$ and $h^2(X, S^{4k+3l-3n-2}T_X\otimes ((k+l-1)K_X))$, we need the following theorem
\begin{theo}[\cite{bb}]\label{theo:bb}
Let $X$ be a smooth projective surface in $\PP^N$. Then
\begin{equation}
H^0(X, S^m[\Omega_X^1(1)])=0
\end{equation}
if and only if $X$ is not a quadric.
\end{theo}
Now it is easy to see the following
\begin{lem}\label{lem:1}
$H^2(X, S^{4k+3l-3n-2}T_X\otimes ((k+l-1)K_X))=0$
\end{lem}
\begin{proof}
By Serre duality theorem, we have $$H^2(X, S^{4k+3l-3n-2}T_X\otimes ((k+l-1)K_X))=H^0(X, S^{4k+3l-3n-2}T_X^*\otimes ((2-k-l)K_X))$$
Since $4k+3l\geq 3n+2$, for $n$ big, $2-k-l<0$, so we have proper embedding $$S^{4k+3l-3n-2}T_X\otimes ((k+l-1)K_X))\rightarrow S^{4k+3l-3n-2}[\Omega_X^1(1)]$$. Therefore by theorem \ref{theo:bb}, we have $H^0(X, S^{4k+3l-3n-2}T_X^*\otimes ((2-k-l)K_X))=0$
\end{proof}
\smallskip
About $h^0(X, S^{4k+3l-3n-2}T_X\otimes ((k+l-1)K_X))$, we need to estimate it by restricting it to a selected curve. This will be done in the next section.

\subsection{estimations on curves}\label{subsec:on curve}
We continue to use the identification $$S^{4k+3l-3n-2}T_X\otimes ((k+l-1)K_X))=S^{4k+3l-3n-2}T_X^*\otimes ((3n-3k-2l+1)K_X)$$
By theorem \ref{theo:bb}, when $(3n-3k-2l+1)(d-4)\leq 4k+3l-3n-2$, we have
$$H^0(X, S^{4k+3l-3n-2}T_X^*\otimes ((3n-3k-2l+1)K_X))=0$$
To simplify our notation, we will write $p=4k+3l-3n-2$ and $q=3n-3k-2l+1$.

Now when $q(d-4)>p$, since $\ocal_X(1)$ is very ample, by Bertini's theorem, generic divisor in $\ocal_X((d-4)q-p)$ is an irreducible and smooth curve $C_{p,q}$. By the remark on theorem \ref{theo:f}, when $(d-4)q-p>2d$, we can pick $C_{p,q}$ such that the restriction $T_X^*|_{C_{p,q}}$ is semistable. As the case $0<(d-4)q-p\leq 2d$ will not contribute to the coefficient of $n^5$, we will ignore this situation. From now on we will always assume $(d-4)q-p>2d$

Consider the following exact sequence
\begin{equation}
0\rightarrow S^p[T_X^*(1)]\rightarrow S^pT_X^*\otimes (qK_X)\rightarrow S^pT_X^*\otimes (qK_X)\otimes \ocal_{C_{p,q}} \rightarrow 0
\end{equation}

We can read from the long exact cohomology sequence of the above short exact sequence the following inequality
\begin{equation}\label{for:ine restr to C}
H^0(X, S^pT_X^*\otimes (qK_X))\leq H^0(C_{p,q}, S^pT_X^*\otimes (qK_X)|_{C_{p,q}})
\end{equation}

\smallskip
On $C_{p,q}$, by adjunction formula, we have $K_{C_{p,q}}=(K_X+C_{p,q})|_{C_{p,q}}$
Similar as lemma \ref{lem:1}, we can prove the following lemma
\begin{lem}\label{lem:c}
For $p>0$, $$H^1(C_{p,q}, S^pT_X^*\otimes (qK_X)|_{C_{p,q}})=0$$
\end{lem}
\begin{proof}
Since $H^1(C_{p,q}, S^pT_X^*\otimes (qK_X)|_{C_{p,q}})=H^0(C_{p,q}, (S^pT_X^*\otimes ((q-1)K_X-C_{p,q})^*|_{C_{p,q}})$, where $S^pT_X^*\otimes ((q-1)K_X-C_{p,q}=$
By our choice of $C_{p,q}$, the restriction $S^pT_X^*|_{C_{p,q}}$ is semistable. And the assumptions on $p$ and $q$ implies $\deg (S^pT_X^*\otimes ((q-1)K_X-C_{p,q}))^*<0$, therefore by proposition \ref{theo:ko} and \ref{theo:ko2} again, $$H^0(C_{p,q}, (S^pT_X^*\otimes ((q-1)K_X))^*|_{C_{p,q}})=0$$
\end{proof}

Since again the case $p=0$ will not contribute to the coefficient of $n^5$, we also ignore this situation.

Now by lemma \ref{lem:c}, $$H^0(C_{p,q}, S^pT_X^*\otimes (qK_X)|_{C_{p,q}})=\chi(S^pT_X^*\otimes (qK_X)|_{C_{p,q}})$$

And since $\rank S^pT_X^*=p+1 $ and $c_1(S^pT_X^*\otimes (qK_X))=\frac{p(p+1)}{2}c_1(T_X^*)+(p+1)qc_1(K_X)$ by Hirzebruch-Riemann-Roch formula
\begin{eqnarray}
\chi(S^pT_X^*\otimes (qK_X)|_{C_{p,q}})=\deg_{C_{p,q}}(\frac{p(p+1)}{2}c_1(T_X^*)+(p+1)qc_1(K_X)-\frac{p+1}{2}c_1(K_{C_{p,q}}))\\
=(p+1)[(\frac{p}{2}+q)d(d-4)-\frac{1}{2}((q+1)(d-4)-p)d](q(d-4)-p)
\end{eqnarray}

Again, we only care about terms that will contribute to the coefficient of $n^5$, the equation above can be simplified to
\begin{equation}\label{for:on C}\chi(S^pT_X^*\otimes (qK_X)|_{C_{p,q}})\sim \frac{pq(p+q)}{2}d(d-4)^2-\frac{p^3}{2}d(d-3)\end{equation}

Now combining theorem \ref{theo:chi}, proposition \ref{prop:h2} and formula \ref{for:on C}, we get the following estimation
\begin{theo}
\label{theo:mainestimation}
$$h^0(X_3, \ocal_3(2n,n))\geq (\frac{159}{512}c_1^2-\frac{27}{160}c_2)n^5-\sum_{(k,l)\in S}[\frac{pq(p+q)}{2}d(d-4)^2-\frac{p^3}{2}d(d-3)]+O(n^4)
$$
where $p=4k+3l-3n-2$ and $q=3n-3k-2l+1$ and the set $$S=\{0\leq k\leq n, 0\leq l\leq 3n-3k, 4k+3l-3n-2\geq 0,(d-4)q-p>2d\}$$
\end{theo}

We can make the set $S$ bigger for the summation of $\frac{pq(p+q)}{2}d(d-4)^2$ by just requiring $q>0$ and the set $S$ smaller for the summation of $\frac{p^3}{2}d(d-3)$ by requiring $q\geq p$. Then we get the following estimation
\begin{cor}When $d\geq 5$,
\begin{equation}h^0(X_3, \ocal_3(2n,n))\geq d(\frac{33}{320}d^2-\frac{359523951}{240100000}d+\frac{799455603}{240100000})n^5+O(n^4)
\end{equation}
\end{cor}
When $d\geq 12$, the coefficient of $n^5$ is positive,
therefore we get the following
\begin{cor}
When $d\geq 12$, $\ocal_3(2n,n)$ is big.
\end{cor}
\begin{rem}
By lemma \ref{lem:c}, when $p>0$ and $q(d-4)>p$, we have the following exact sequence
\begin{eqnarray}
0\rightarrow H^0(X, S^pT_X^*\otimes (qK_X))\rightarrow H^0(C_{p,q}, S^pT_X^*\otimes (qK_X)\otimes \ocal_{C_{p,q}})\\ \rightarrow H^1(X, S^p[T_X^*(1)]) \rightarrow  H^1(X, S^pT_X^*\otimes (qK_X))\rightarrow 0
\end{eqnarray}
So one can also use $h^1(X, S^p[T_X^*(1)])$ to estimate $h^1(X, S^pT_X^*\otimes (qK_X))$ directly instead of using $\chi(S^pT_X^*\otimes (qK_X))$ and $h^0(C_{p,q}, S^pT_X^*\otimes (qK_X)\otimes \ocal_{C_{p,q}})$. But that estimation would be worse than the one we got above.
\end{rem}

\subsection{$\ocal_3(bn,n)$}
The following formulas follow directly from the analysis in the last section

We write $c=b+1$ for $b\geq 2$, $c\geq 3$.

First we have
\begin{equation}
\chi(\ocal_3(bn,n))=\sum_{k=0}^n\sum_{l=0}^{cn-3k}\chi(\ocal_1(cn-4k-3l)\otimes \pi_1^*(K_X)^{\otimes (k+l)})
\end{equation}

Therefore we have
$$\chi(\ocal_3(bn,n))=n^5[(\frac{5}{24}c^4-c^3+\frac{7}{3}c^2-\frac{31}{12}c+\frac{41}{40})c_2-(\frac{1}{24}c^4-\frac{1}{6}c^3+\frac{1}{3}c^2-\frac{1}{3}c+\frac{1}{8})c_1^2]+O(n^4)$$

Next, similar to proposition \ref{prop:h2} we have

\begin{prop}\label{prop:h211}
For $c>4$, we have
\begin{eqnarray}
\sum_{\begin{array}{clcr}0\leq k\leq n,\quad 0\leq l\leq cn-3k\\cn-4k-3l\leq -2
\end{array}}\chi(S^{4k+3l-cn-2}T_X^*\otimes ((cn-3k-2l+1)K_X))\\
=(f_1c_1^2+f_2c_2)n^5+O(n^4) \qquad \qquad \qquad
\end{eqnarray}
where $$f_1=\frac{5}{81}c^4-\frac{47}{162}c^3+\frac{229}{324}c^2-\frac{145}{162}c+\frac{305}{648}$$
$$f_2=-\frac{2}{9}c^4+\frac{10}{9}c^3-\frac{25}{9}c^2+\frac{125}{36}c-\frac{125}{72}$$
\end{prop}

Now similar to theorem \ref{theo:mainestimation}, we have
\begin{theo}
\label{theo:mainestimation11}
For $c>4$,
$$h^0(X_3, \ocal_3(bn,n))\geq (g_1c_1^2+g_2c_2)n^5-\sum_{(k,l)\in S'}[\frac{pq(p+q)}{2}d(d-4)^2-\frac{p^3}{2}d(d-3)]+O(n^4)
$$
where $$g_1=\frac{13}{648}c^4-\frac{10}{81}c^3+\frac{121}{324}c^2-\frac{91}{162}c+\frac{28}{81}$$
$$g_2=-\frac{1}{72}c^4+\frac{1}{9}c^3-\frac{4}{9}c^2+\frac{8}{9}c-\frac{32}{45}$$
and $p=4k+3l-cn-2$ , $q=cn-3k-2l+1$ and the set $$S'=\{0\leq k\leq n, 0\leq l\leq 3n-3k, 4k+3l-cn-2\geq 0,(d-4)q-p>2d\}$$
\end{theo}

Now let $d=11$, we can calculate the second summand on the right hand side of the formula, which is
$$\frac{290521}{795906}(c^4-2c^3+2c^2-c+\frac{1}{5})$$

\begin{theo} For $d=11$, $$h^0(X_3, \ocal_3(bn,n))\geq yn^5+O(n^4)$$
where $$y=\frac{1}{44217}[-\frac{394823}{4}c^4+1575508c^3-\frac{36295897}{4}c^2+22513040c-\frac{40944629}{2}]$$
\end{theo}
\begin{figure}\label{fig:y}
     \includegraphics{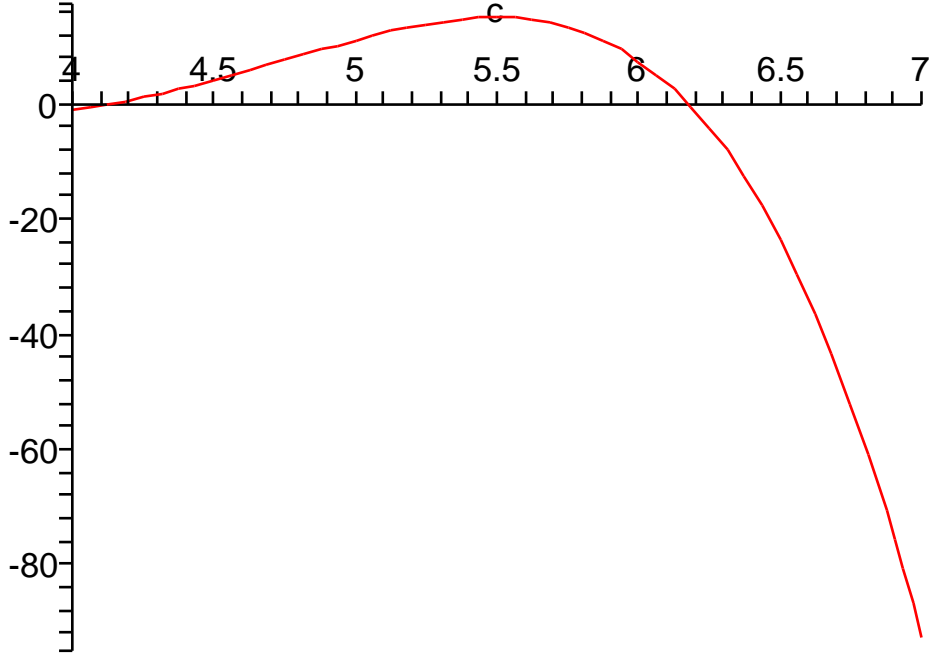}
     \caption{graph of y for $4\leq c\leq 7$}
   \end{figure}

 One can see directly from graph \ref{fig:y}, that there exists $c\in (4,7)$ such that $y>0$, therefore we have
 \begin{theo}
 For $d=11$, $\ocal_3(1)$ is big.
 \end{theo}
\section{The Fourth Semple Jet Bundle}\label{sec:4sem}

\subsection{Euler Characteristic}As the notations we use on the third Semple jet bundle, on $X_4$ we will denote by
 $$\ocal_4(a_2,a_3,a_4)=\pi^*_{4,2}\ocal_2(a_2)\otimes \pi^*_{4,3}\ocal_3(a_3)\otimes \ocal_4(a_4),\qquad \ocal_4(a_3, a_4)=\pi^*_4\ocal_3(a_3)\otimes \ocal_4(a_4))$$
 To show that $\ocal_4(1)$ is big, instead of showing that $\ocal_4(2,1)$ is big as on $X_3$, we will show that $\ocal_4(6,2,1)$ is big for $d\geq 10$.
First we need to calculate $\chi(X_4, \ocal_4(6n,2n,n))$. By pushing forward to $X_2$, we get a formula similar to the one in proposition \ref{for:chi3}
\begin{prop}
\begin{equation}\label{for:chi4}
\chi(\ocal_4(6n,2n,n))=\sum_{k=0}^n\sum_{l=0}^{3n-3k}\chi(\ocal_2(9n-4k-3l)\otimes \pi_2^*(\det V_1^*)^{\otimes (k+l)})
\end{equation}
\end{prop}
\begin{proof}
Just repeat the arguments on $X_3$.
\end{proof}
Therefore to calculate $\chi(\ocal_4(6n,2n,n))$, we need to use the Hirzebruch-Riemann-roch  formula for line bundles on $X_2$. Denote by $\mathfrak{L}_{k,l}$ the line bundle $\ocal_2(9n-4k-3l)\otimes \pi_2^*(\det V_1^*)^{\otimes (k+l)}$.
Again, we just need to calculate the term $\frac{1}{24}c_1(\mathfrak{L}_{k,l})^4$ in the expression of $\chi(X_4,\mathfrak{L}_{k,l})$.

By formula \ref{for:det}, we have $$c_1(\det V_1^*)=-c_1-u_1$$
hence $$c_1(\mathfrak{L}_{k,l})=(9n-4k-3l)u_2-(k+l)(c_1+u_1)$$
By expanding out the expression of $c_1(\mathfrak{L}_{k,l})^4$, and plugging the equations
\begin{eqnarray}
u_2^2+c_1(V_1)u_2+c_2(V_1)=0\\u_1^2+c_1u_1+c_2=0
\end{eqnarray}
where $c_1(V_1)=c_1+u_1$ and $$c_2(V_1)=-u_1c_1(T_{X_1|X})=-u_1(c_1+2u_1)=c_1u_1+2c_2$$
we get

\begin{equation}
c_1(\mathfrak{L}_{k,l})^4=f(k,l)c_2\\ -g(k,l)c_1^2
\end{equation}
where $$f(k,l)=5(9n-4k-3l)^4+12(9n-4k-3l)^3(k+l)+\\6(9n-4k-3l)^2(k+l)^2+4(9n-4k-3l)(k+l)^3$$
and $$g(k,l)=4(9n-4k-3l)^3(k+l)\\+6(9n-4k-3l)^2(k+l)^2+4(9n-4k-3l)(k+l)^3+(9n-4k-3l)^4$$

Summing up, we get
\begin{prop}
\begin{equation}
 \chi(\ocal_4(6n,2n,n))=(\frac{23629}{60}c_2-\frac{1213}{12}c_1^2)n^6+O(n^5)
\end{equation}
\end{prop}

\begin{rem}
Using similar formula, we can also calculate $\chi(\ocal_4(2n,n))$, actually
$$ \chi(\ocal_4(2n,n))=(\frac{1}{6}c_1^2-\frac{61}{30}c_2)n^6+O(n^5)$$
which is clearly negative for $n$ big. This explains why we do not use $\ocal_4(2n,n)$
\end{rem}
Now that we have $\chi$, by the equation $\chi=h^0-h^1+h^2-h^3+h^4-h^5+h^6$ on $X_4$, to have
an estimation of $h^0$, we need to estimate $h^2$, $h^4$ and $h^6$.

First, since $R^i\pi_{4*}(\ocal_4(6n,2n,n))=0$ for $n\geq -1$ and $i>0$, for $n$ big, we always have
$$h^6(X_4, \ocal_4(6n,2n,n))=h^6(X_3, \pi_{4*}\ocal_4(6n,2n,n))=0$$

Next, since we have a filtration of $\pi_{4*}\ocal_4(6n,2n,n)=\pi_3^*(\ocal_2(6n))\otimes S^{n}V_3^*\otimes \ocal_3(2n)$ with graded bundle
$$Gr^{\bullet}(\pi_3^*(\ocal_2(6n))\otimes S^{n}V_3^*\otimes \ocal_3(2n))=\bigoplus_{0\leq k \leq n}\ocal_3(3n-3k)\otimes \pi_3^*(\det V_2^*)^{\otimes k}\otimes \pi_3^*(\ocal_2(6n)$$
Since $3n-3k\geq 0$, we have vanishing higher direct images, therefore we can calculate the cohomology of the graded pieces by pushing forward to $X_2$, thus we have the following inequalities
\begin{equation}
h^2(X_4, \ocal_4(6n,2n,n))\leq \sum_{0\leq k\leq n}h^2(X_2,S^{3n-3k}V_2^*\otimes (\det V_2^*)^{\otimes k}\otimes \ocal_2(6n))
\end{equation}
\begin{equation}
h^4(X_4, \ocal_4(6n,2n,n))\leq \sum_{0\leq k\leq n}h^4(X_2,S^{3n-3k}V_2^*\otimes (\det V_2^*)^{\otimes k}\otimes \ocal_2(6n))
\end{equation}
We use the filtration of $S^{3n-3k}V_2^*\otimes (\det V_2^*)^{\otimes k}\otimes \ocal_2(6n)$ with graded bundle
$$Gr^{\bullet}(S^{3n-3k}V_2^*\otimes (\det V_2^*)^{\otimes k}\otimes \ocal_2(6n))=\bigoplus_{0\leq l \leq 3n-3k}\ocal_2(9n-4k-3l)\otimes \pi_2^*(\det V_1^*)^{\otimes (k+l)}$$

Notice that $9n-4k-3l\geq 0$ for $0\leq k \leq n$ and $0\leq l \leq 3n-3k$, again we get vanishing higher direct images of the graded pieces under the projection $\pi_2: X_2\rightarrow X_1$. Therefore we can push $\ocal_4(6n,2n,n)$ further onto $X_1$. In particular since $\dim X_1=3$, we have $$h^4(X_2, \ocal_2(9n-4k-3l)\otimes \pi_2^*(\det V_1^*)^{\otimes (k+l)})=h^4(X_1, S^{9n-4k-3l}V_1^*\otimes (\det V_1^*)^{\otimes (k+l)})=0$$
therefore we have $h^4(X_2,S^{3n-3k}V_2^*\otimes (\det V_2^*)^{\otimes k}\otimes \ocal_2(6n))=0$, hence
\begin{equation}h^4(X_4, \ocal_4(6n,2n,n))=0\end{equation}
\subsection{Estimation of $H^2$}
For $h^2$, we only have the inequality

\begin{equation}
h^2(X_4, \ocal_4(6n,2n,n))\leq \sum_{0\leq k\leq n}\sum_{0\leq l \leq 3n-3k}h^2(X_1, S^{9n-4k-3l}V_1^*\otimes (\det V_1^*)^{\otimes (k+l)})
\end{equation}

\smallskip

To simplify our notations, we denote $p=9n-4k-3l$ and $q=k+l$, then for $S^p V_1^*\otimes (\det V_1^*)^{\otimes q}$, we have the filtration with graded bundle
$$Gr^{\bullet}(S^p V_1^*\otimes (\det V_1^*)^{\otimes q})=\bigoplus_{0\leq j \leq p}\ocal_1(p-3j-q)\otimes ((j+q)K_X)$$
First we have the following
\begin{prop}When $p-3j-q\geq -1$, we have $h^2(X_1, \ocal_1(p-3j-q)\otimes ((j+q)K_X))=0$

\end{prop}
\begin{proof}
Under the assumption, again we have vanishing higher direct images of $\ocal_1(p-3j-q)\otimes ((j+q)K_X)$ under the projection $\pi_1: X_1\rightarrow X$. When $p-3j-q=-1$, $\pi_{1*}\ocal_1(p-3j-q)\otimes ((j+q)K_X)=0$, hence the conclusion in this case.

When $p-3j-q\geq 0$, $\pi_{1*}(\ocal_1(p-3j-q)\otimes ((j+q)K_X))=S^{p-3j-q}T_X^*\otimes ((j+q)K_X)$.

By Serre duality theorem, $$H^2(X, S^{p-3j-q}T_X^*\otimes ((j+q)K_X))\cong H^0(X, S^{p-3j-q}T_X\otimes ((1-j-q)K_X))$$
Using the identity $S^{p-3j-q}T_X\otimes ((1-j-q)K_X))=S^{p-3j-q}T_X^*\otimes ((1-j-q-(p-3j-q))K_X))$, and theorem \ref{theo:bb}, it is easy to see that $H^0(X, S^{p-3j-q}T_X\otimes ((1-j-q)K_X))=0$. Therefore
$$h^2(X_1, \ocal_1(p-3j-q)\otimes ((j+q)K_X))=h^2(X, S^{p-3j-q}T_X^*\otimes ((j+q)K_X))=0$$
\end{proof}

When $p-3j-q\leq -2$, by the Leray spectral sequence of the projection $\pi_1:X_1\rightarrow X$ again ,
$$
h^2(X_1, \ocal_1(p-3j-q)\otimes \pi_1^*(K_X)^{\otimes (j+q)})=h^1(X, R^1\pi_{1*}(\ocal_1(p-3j-q)\otimes \pi_1^*(K_X)^{\otimes (j+q)}))
$$
By proposition \ref{for:R1}, we have $$R^1\pi_{1*}(\ocal_1(p-3j-q)\otimes \pi_1^*(K_X)^{\otimes (j+q)})=S^{3j+q-p-2}T_X\otimes ((j+q-1)K_X)$$
Further more, by Serre duality on $X$, we have $$h^1(X, S^{3j+q-p-2}T_X\otimes ((j+q-1)K_X))=h^1(X, S^{3j+q-p-2}T_X^*\otimes ((2-j-q)K_X))$$
Thus we have the following equation
\begin{equation}
h^2(X_1, \ocal_1(p-3j-q)\otimes \pi_1^*(K_X)^{\otimes (j+q)})=h^1(X, S^{3j+q-p-2}T_X^*\otimes ((2-j-q)K_X))
\end{equation}

We want to use $\chi(X, S^{3j+q-p-2}T_X^*\otimes ((2-j-q)K_X))$ to estimate $h^2(X_1, \ocal_1(p-3j-q)\otimes \pi_1^*(K_X)^{\otimes (j+q)})$.

\begin{prop}
\begin{eqnarray}
\sum_{\begin{array}{clcr}0\leq k\leq n,\quad 0\leq l\leq 3n-3k\\0\leq j\leq p,\quad 3j+q-p-2\geq 0
\end{array}}\chi(X, S^{3j+q-p-2}T_X^*\otimes ((2-j-q)K_X))\\
=(\frac{3617245553}{28449792}c_1^2-\frac{8184073}{20160}c_2)n^6+O(n^5)   \qquad \qquad \qquad
\end{eqnarray}
\end{prop}

\smallskip

By theorem \ref{theo:bb}, for $j+q>2$
$$h^0(X, S^{3j+q-p-2}T_X^*\otimes ((2-j-q)K_X))=0$$so $\chi(X, \mathfrak{F}_{j,p,q})=h^2(X, \mathfrak{F}_{j,p,q})-h^1(X, \mathfrak{F}_{j,p,q})$, where $\mathfrak{F}_{j,p,q}=S^{3j+q-p-2}T_X^*\otimes ((2-j-q)K_X)$. We will ignore the case $j+q\leq 2$.

To estimate $h^2(X, \mathfrak{F}_{j,p,q})=h^0(X, S^{3j+q-p-2}T_X\otimes ((j+q-1)K_X))=h^0(X, S^{3j+q-p-2}T_X^*\otimes ((p+1-2j)K_X))$, as in the estimation of $X_3(2n,n)$, we need to restrict $\mathfrak{F}_{j,p,q}$ to $C_{\alpha,\beta}$, where $\alpha=3j+q-p-2$ and $\beta=p+1-2j$. More precisely, for $(d-4)\beta-\alpha>2d$, $$C_{\alpha,\beta}\in |\ocal_{\CP^3}((d-4)\beta-\alpha)\otimes \ocal_X|$$ is chosen as in section \ref{subsec:on curve}.

Same as formula \ref{for:ine restr to C}, lemma \ref{lem:c} and formula \ref{for:on C} in section \ref{subsec:on curve} we have
\begin{equation}\label{for:4 ine restr to C}
H^0(X, S^{\alpha}T_X^*\otimes (\beta K_X))\leq H^0(C_{\alpha,\beta}, S^{\alpha}T_X^*\otimes (\beta K_X)|_{C_{\alpha,\beta}})
\end{equation}

\begin{equation}\label{for:4 on C}\chi(S^{\alpha}T_X^*\otimes (\beta K_X)|_{C_{\alpha,\beta}})\sim \frac{\alpha\beta(\alpha+\beta)}{2}d(d-4)^2-\frac{{\alpha}^3}{2}d(d-3)\end{equation}
and
\begin{lem}\label{lem:4 c}
For $\alpha>0$, $$H^1(C_{\alpha,\beta}, S^{\alpha}T_X^*\otimes (\beta K_X)|_{C_{\alpha,\beta}})=0$$
\end{lem}
Then we have the following theorem
\begin{theo}
\begin{equation}h^0(X_4, \ocal_4(6n,2n,n))\geq f(d)n^6+O(n^5)
\end{equation}
where $$f(d)=d(\frac{18461}{1920}d^2-\frac{41723445050414378269}{345738849132600000}d+\frac{90181735116469021057}{345738849132600000})$$
\end{theo}
\begin{proof}
By the arguments above, first we have
\begin{eqnarray}
&h^0(X_4,& \ocal_4(6n,2n,n))\geq \chi(X_4, \ocal_4(6n,2n,n))\\&+&\sum_{\begin{array}{clcr}0\leq k\leq n,\quad 0\leq l\leq 3n-3k\\0\leq j\leq p,\quad 3j+q-p-2\geq 0
\end{array}}\chi(X, S^{3j+q-p-2}T_X^*\otimes ((2-j-q)K_X))\\&-&\sum_{(j,k,l)\in I}(\frac{\alpha\beta(\alpha+\beta)}{2}d(d-4)^2-\frac{{\alpha}^3}{2}d(d-3))\qquad \qquad\qquad\qquad
\end{eqnarray}
where the set $I$ is defined by $$I=\{j,k,l|0\leq j\leq p,0\leq k\leq n,0\leq l\leq 3n-3k, (d-4)\beta> \alpha, \alpha>0\}$$
We already have the results of the first two sums on the right side. In the last sum of the left side, for $d\geq 10$, for the first term we can make the set $I$ bigger by just requiring $\beta>0$ and for the second term we can make the set $I$ smaller by requiring $(10-4)\beta>\alpha$, then we get the conclusion.
\end{proof}
\begin{rem}
When $d=9$, we can calculate the summation directly and get $$h^0(X_4, \ocal_4(6n,2n,n))\geq-304.5398797n^6+O(n^5)$$ therefore no conclusion can be made in this case.
\end{rem}

Since for $d\geq 10$, $f(d)>0$, we have
\begin{cor}
$\ocal_{X_4}(1)$ is big for $d\geq 10$.
\end{cor}

\end{document}